\documentclass{ws-ijgmmp}

\usepackage[shortlabels]{enumitem}
\usepackage{ifpdf}
\ifpdf
 \usepackage[colorlinks,final, backref=page, hyperindex]{hyperref}
\else
  \usepackage[colorlinks,final,backref=page,hyperindex,hypertex]{hyperref}
\fi

\usepackage{tikz}
\usepackage[active]{srcltx}

\newtheorem{thm}{Theorem}[section]
\newtheorem{lem}[thm]{Lemma}
\newtheorem{cor}[thm]{Corollary}
\newtheorem{pro}[thm]{Proposition}
\newtheorem{ex}[thm]{Example}

\newtheorem{defi}[thm]{Definition}
\newtheorem{case}{Case}
\numberwithin{equation}{section}

\newcommand{\nc}{\newcommand}
\newcommand{\delete}[1]{}

\nc{\mlabel}[1]{\label{#1}}  
\nc{\mcite}[1]{\cite{#1}}  
\nc{\mref}[1]{\ref{#1}}  
\nc{\mbibitem}[1]{\bibitem{#1}} 

\delete{
\nc{\mlabel}[1]{\label{#1}{\hfill \hspace{1cm}{\bf{{\ }\hfill(#1)}}}}
\nc{\mcite}[1]{\cite{#1}{{\bf{{\ }(#1)}}}}  
\nc{\mref}[1]{\ref{#1}{{\bf{{\ }(#1)}}}}  
\nc{\mbibitem}[1]{\bibitem[\bf #1]{#1}} 
}

\newcommand {\emptycomment}[1]{}

\delete{
\setlength{\baselineskip}{1.8\baselineskip}

\newcommand{\emptycomment}[1]{}

}

\def\sl{\mathfrak{sl}}

\def\al{\alpha}
\def\be{\beta}

\def\hh{\mathfrak{h}}
\def\gg{\mathfrak{g}}
\def\ff{\mathfrak{f}}
\def\ee{\mathfrak{e}}

\def\bb{\mathfrak{b}}

\def\oo{\mathfrak{o}}
\def\sp{\mathfrak{sp}}
\def \<{\langle}
\def \>{\rangle}

\def\GG{\mathcal{G}}

\newcommand{\C}{\mathbb {C}}

\newcommand{\N}{\mathbb{N}}
\newcommand{\Z}{\mathbb{Z}}

\def\V{\mathrm{V}}

\allowdisplaybreaks

\begin{document}

\markboth{C. Bai and D. Gao}
{Compatible root graded anti-pre-Lie algebraic structures}

%
\catchline{}{}{}{}{}
%

\title{Compatible root graded anti-pre-Lie algebraic structures on finite-dimensional complex simple Lie algebras
}

\author{Chengming Bai}

\address{Chern Institute of Mathematics and LPMC, Nankai University,\\ Tianjin 300071, P. R. China\\
\email{baicm@nankai.edu.cn} }

\author{Dongfang Gao\footnote{
Corresponding author.}}

\address{Chern Institute of Mathematics and LPMC, Nankai University, \\Tianjin 300071, P. R. China,\\ 
and Institut Camille Jordan, Universit\'{e} Claude Bernard Lyon 1,\\
  Lyon, 69622, France     \\
gao@math.univ-lyon1.fr}

\maketitle


\begin{abstract}
We investigate the compatible root graded anti-pre-Lie algebraic
structures on any finite-dimensional complex simple Lie algebra by
the representation theory of ${\rm sl_2(\C)}$. We show that there
does not exist a compatible root graded anti-pre-Lie algebraic
structure on a finite-dimensional complex simple Lie algebra
except ${\rm sl_2(\C)}$, whereas there is exactly one compatible
root graded anti-pre-Lie algebraic structure on ${\rm sl_2(\C)}$.
\end{abstract}

\keywords{anti-pre-Lie algebra; simple Lie algebra; root; irreducible representation.}

\ccode{Mathematics Subject Classification 2020: 17B10, 17B65, 17B66, 17B68, 17B70}

\section{Introduction}

It is well-known that symplectic forms on Lie algebras give
compatible pre-Lie algebraic structures on the Lie algebras
themselves \cite{Bai}, that is, pre-Lie algebras are regarded as
the underlying algebraic structures of symplectic forms on Lie
algebras. The ``symmetric" version of a symplectic form on a Lie
algebra is a nondegenerate commutative 2-cocycle \cite{Dzh}.
Accordingly, the notion of anti-pre-Lie algebras was introduced in
\cite{LB} as the underlying algebraic structures of nondegenerate
commutative 2-cocycles on Lie algebras. Moreover, anti-pre-Lie
algebras can be regarded as the ``anti-structures" of pre-Lie
algebras since anti-pre-Lie algebras are characterized as
Lie-admissible algebras whose negative left multiplication
operators give representations of the commutator Lie algebras,
whereas pre-Lie algebras are characterized as Lie-admissible
algebras whose left multiplication operators give representations
of the commutator Lie algebras. Note that pre-Lie algebras arose
from the study of deformations of associative algebras
\cite{Ger}, affine manifolds and affine structures on Lie groups
\cite{Kos} and convex homogeneous cones \cite{Vin}, and
appeared in many fields of mathematics and mathematical physics
(\cite{Bai,  Bur} and the references therein), whereas anti-pre-Lie
algebras are related to a lot of algebraic structures such as
transposed Poisson algebras \cite{BBGW,  DFK}, differential
algebras \cite{CGWZ} and anti-dendriform algebras \cite{GLB}.

Finite-dimensional simple Lie algebras over the complex number
field $\C$ were classified by Cartan and Killing during the decade
1890-1900. There are four families of classical Lie algebras:
$A_n$ $(n\geq 1)$, $B_n$ $(n\geq 2)$, $C_n$ $(n\geq 3)$ and $D_n$
$(n\geq 4)$.  In addition, there are five exceptional Lie
algebras: $E_6$, $E_7$, $E_8$, $F_4$ and $G_2$. It is known that
there does not exist a compatible pre-Lie algebraic structure on
any finite-dimensional complex simple Lie algebra \cite{Med}.
However, it is different in the case of anti-pre-Lie algebras
since there is a compatible anti-pre-Lie algebraic structure on
${\rm sl}_2(\C)$ which has been given in \cite{LB}. So it is
natural to consider whether there are compatible anti-pre-Lie
algebraic structures on any finite-dimensional complex simple Lie
algebra and if the answer is positive, then consider their
classification.

Unfortunately, it seems still hard and complicated to give a
complete classification of compatible anti-pre-Lie algebraic
structures on any finite-dimensional complex simple Lie algebra.
Note that there is a root space decomposition for a
finite-dimensional complex simple Lie algebra. Hence for the
compatible anti-pre-Lie algebraic structures, there is an
additionally natural condition which is ``consistent" with the
root space decomposition, that is, they are the  compatible root
graded anti-pre-Lie algebraic structures. Note that the compatible
anti-pre-Lie algebraic structure on ${\rm sl}_2(\C)$ given in
\cite{LB} is root graded. Therefore in this paper, we study and
classify the compatible root graded anti-pre-Lie algebraic
structures on finite-dimensional complex simple Lie algebras.

 The paper is organized as follows.
 In Section 2, we recall some notions and basic results on anti-pre-Lie algebras, finite-dimensional complex simple Lie algebras
 and the representation theory of ${\rm sl_2(\C)}$.
 In Section 3, we show that there exists exactly one compatible root graded anti-pre-Lie algebraic structure on ${\rm sl_2(\C)}$ which is the example given in \cite{LB}.
In Section 4, we prove that there does not exist a compatible root
graded anti-pre-Lie algebraic structure on a finite-dimensional complex simple Lie
algebra except  ${\rm sl_2(\C)}$.

Throughout this paper, we denote by $\Z,\Z_+,\N, \C$ and $\C^*$ the
set of integers, positive integers, non-negative integers, complex numbers and nonzero complex numbers respectively.
All vector spaces and algebras are over $\C$, unless otherwise stated.

\section{Preliminaries}
 We recall some notions and results on anti-pre-Lie algebras, finite-dimensional complex simple Lie algebras and the representation theory of ${\rm sl_2(\C)}$ for future convenience.
\begin{defi}[\cite{LB}]
An {\bf anti-pre-Lie algebra} is a vector space $A$ with a binary
operation $\circ$ satisfying the following two equations.
\begin{align}
&x\circ (y\circ z)-y\circ (x\circ z)=[y,x]\circ z,\label{anti-1}\\
&[[x,y], z]+[[y,z], x]+[[z,x], y]=0,\label{anti-2}
\end{align}
where $[x,y]=x\circ y-y\circ x$ for any $x,y,z\in A$.
\end{defi}

\begin{lem}[\cite{LB}]\label{property-1}
Let $(A, \circ)$ be an anti-pre-Lie algebra.
Then the following results hold.
\begin{enumerate}
\item The commutator
$$[x,y]=x\circ y-y\circ x, \ \ \  x,y\in A,$$
defines a Lie algebra, denoted by $\GG(A)$, which is called the
{\bf sub-adjacent Lie algebra} of $(A,\circ)$. Furthermore,
$(A,\circ)$ is called a {\bf compatible anti-pre-Lie algebraic structure} on the Lie algebra  $\GG(A)$.
\item Let $\rho: \GG(A)\rightarrow \frak g\frak l(A)$ be a linear map defined by
$\rho(x)=-L_x$ for any $x\in\GG(A)$, where $-L_x$ is the negative left multiplication
operator, that is, $-L_x(y)=-x\circ y$ for any $y\in A$. Then
$\rho$ defines a representation of the Lie algebra $\GG(A)$.
\end{enumerate}
\end{lem}

Let $${\rm sl_2(\C)}={\rm span}\{e_{12}=\begin{pmatrix}
0&1\\
0&0
\end{pmatrix},
e_{21}=\begin{pmatrix}
0&0\\
1&0
\end{pmatrix},
h_1=\begin{pmatrix}
1&0\\
0&-1
\end{pmatrix} \},$$ denote the 3-dimensional simple Lie algebra satisfying the following Lie brackets.
$$[h_1, e_{12}]=2e_{12},\ \ \ [h_1, e_{21}]=-2e_{21},\ \ \ [e_{12}, e_{21}]=h_1.$$

\begin{ex}\label{ex:LB}(\cite[Example 2.21]{LB}).
Let $\circ$ be the binary operation  on ${\rm sl_2(\C)}$
defined by
\begin{align}
&h_1\circ e_{12}=-2e_{12}, \ e_{12}\circ h_1=-4e_{12},\ h_1\circ e_{21}=2e_{21},\  e_{21}\circ h_1=4e_{21},  \label{sl2-graded-1}\\
&e_{12}\circ e_{21}=\frac{1}{2}h_1, \  e_{21}\circ e_{12}=-\frac{1}{2}h_1,\ h_1\circ h_1=e_{12}\circ e_{12}=e_{21}\circ e_{21}=0.\label{sl2-graded-2}
\end{align}
Then $({\rm sl_2(\C)}, \circ)$ is a compatible anti-pre-Lie
algebraic structure on the Lie algebra ${\rm sl_2(\C)}$.
\end{ex}

Suppose that $\L$ is a finite-dimensional complex simple Lie
algebra. Then there exists a nilpotent self-normalizing subalgebra
$\hh$ of $\L$, called the {\bf  Cartan subalgebra} of $\L$.
Moreover,  $\L$ has the following {\bf root space decomposition}
(see \cite{Car, Hum}).
$$\L=\L_0\oplus\oplus_{\delta\in\Phi}\L_\delta,$$
where $\L_0=\hh$ is the Cartan subalgebra, $\Phi\subseteq \hh^*$ is the {\bf set of roots} of $\L$ and
$$\L_\delta=\{x\in\L~|~[h, x]=\delta(h)x,  h\in\hh\},\ \ \  \delta\in\Phi,$$ is the {\bf root space}.

\begin{defi}
Let $\L$ be a finite-dimensional complex simple Lie algebra with
the root space decomposition
$\L=\L_0\oplus\oplus_{\delta\in\Phi}\L_\delta$. Then a compatible
anti-pre-Lie algebraic structure $(\L,\circ)$ on $\L$ is called
{\bf root graded} if $L_{\delta_1}\circ L_{\delta_2}\subseteq
L_{\delta_1+\delta_2}$ for any $\delta_1,\delta_2\in \Phi_0$,
where $\Phi_0=\Phi\cup \{0\}$. In this case, $(\L,\circ)$ is
called a {\bf compatible root graded anti-pre-Lie algebraic
structure} on $\L$.
\end{defi}

\begin{ex}\label{graded-on-sl2}
Note that ${\rm sl_2(\C)}={\rm span}\{e_{12}=\begin{pmatrix}
0&1\\
0&0
\end{pmatrix},
e_{21}=\begin{pmatrix}
0&0\\
1&0
\end{pmatrix},
 h_1=\begin{pmatrix}
1&0\\
0&-1
\end{pmatrix}\}$ has the following root space decomposition.
$${\rm sl_2(\C)}={\rm sl_2(\C)}_0\oplus{\rm sl_2(\C)}_\delta\oplus{\rm sl_2(\C)}_{-\delta},$$
where ${\rm sl_2(\C)}_0=\C h_1$ is the Cartan subalgebra of ${\rm
sl_2(\C)}, {\rm sl_2(\C)}_\delta=\C e_{12}, {\rm
sl_2(\C)}_{-\delta}=\C e_{21}$ and $\delta:\C h_1\rightarrow \C$
is a linear map defined by $\delta(h_1)=2$. Hence by
Eqs.~\eqref{sl2-graded-1} and \eqref{sl2-graded-2}, the
anti-pre-Lie algebra in Example~\ref{ex:LB} is a compatible root
graded anti-pre-Lie algebraic structure on ${\rm sl_2(\C)}$.
\end{ex}

\begin{defi}
Let $V$ be a representation of ${\rm sl_2(\C)}$. Then $V$ is
called a {\bf weight representation} of ${\rm sl_2(\C)}$ if
$V=\oplus_{\lambda\in\C}V_\lambda$, where $V_\lambda=\{v\in
V~|~h_1.v=\lambda v\}$. Furthermore, $V_\lambda$ is called a {\bf
weight space} of weight $\lambda$ and the {\bf weight set} of $V$
is the set consisting of $\lambda$ with $V_\lambda\ne 0$.
\end{defi}

\begin{defi}
Let $V=\oplus_{\lambda\in\C}V_\lambda$ be a weight representation of ${\rm sl_2(\C)}$.
If a nonzero vector $v\in V_\lambda$ satisfies $e_{12}.v=0$, then $v$ is called a {\bf highest weight vector} of weight $\lambda$.
Similarly, if a nonzero vector $v\in V_\lambda$ satisfies $e_{21}.v=0$, then $v$ is called a {\bf lowest weight vector} of weight $\lambda$.
\end{defi}

\begin{lem}[\cite{Hum, Maz}]\label{representation-1}
For any $m\in\N$, let $\V(m)=\oplus_{i=0}^{m}\C v_i$ be an $(m+1)$-dimensional vector space.
Then $\V(m)$ is an irreducible weight representation of ${\rm sl_2(\C)}$ with the following actions.
$$h_1.v_i=(m-2i)v_i,\ e_{21}.v_i=(i+1)v_{i+1},\ e_{12}.v_i=(m-i+1)v_{i-1}, \  0\le i\le m,$$
where $v_{-1}=v_{m+1}=0$. Moreover, any nonzero weight space of $\V(m)$ is 1-dimensional and the weight set of $\V(m)$ is $\{m, m-2, m-4,\cdots, -m+2, -m\}$.
\end{lem}

\begin{lem}[\cite{Hum, Maz}]\label{representation-2}
Let $m\in\N$ and $V$ be an $(m+1)$-dimensional irreducible representation of ${\rm sl_2(\C)}$.
Then $V$ is isomorphic to $\V(m)$ as representations of ${\rm sl_2(\C)}$.
Moreover, any finite-dimensional representation of ${\rm sl_2(\C)}$ can be decomposed into a direct sum
of finite-dimensional irreducible representations of ${\rm sl_2(\C)}$.
\end{lem}

\section{Compatible root graded anti-pre-Lie algebraic structures on ${\rm sl_2(\C)}$}
We show that there exists exactly one compatible root graded
anti-pre-Lie algebraic structure on ${\rm sl_2(\C)}$, that is, the
one given in Example~\ref{ex:LB}.

Suppose that $({\rm sl_2(\C)}, \circ)$ is a compatible root graded anti-pre-Lie algebra on $${\rm sl_2(\C)}={\rm span}\{e_{12}=\begin{pmatrix}
0&1\\
0&0
\end{pmatrix},
e_{21}=\begin{pmatrix}
0&0\\
1&0
\end{pmatrix},
 h_1=\begin{pmatrix}
1&0\\
0&-1
\end{pmatrix}\}.$$
By the root space decomposition of  ${\rm sl_2(\C)}$  given in
Example \ref{graded-on-sl2}, we can write that
\begin{align}
   & h_1\circ e_{12}=\al_1e_{12},\ e_{12}\circ h_1=(\al_1-2) e_{12},\label{sl2-1}\\
   &h_1\circ e_{21}=\be_1e_{21},\ e_{21}\circ h_1=(\be_1+2) e_{21},  \label{sl2-2}\\
   &e_{12}\circ e_{21}=\gamma_1h_1,\ e_{21}\circ e_{12}=(\gamma_1-1) h_1,\label{sl2-3}\\
   &h_1\circ h_1=\lambda_1 h_1,\ e_{12}\circ e_{12}=e_{21}\circ e_{21}=0, \label{sl2-4}
\end{align}
where $\al_1,\be_1,\gamma_1,\lambda_1\in\C$.

\begin{lem}\label{lemma-sl2}
With the notations in Eqs.~\eqref{sl2-1}-\eqref{sl2-4},
we have $\al_1\ne 2, \be_1\ne -2$.
\end{lem}
\begin{proof}
Assume that $\al_1=2$. By the following equation
$$e_{12}\circ (e_{21}\circ e_{12})-e_{21}\circ (e_{12}\circ e_{12})=[e_{21},e_{12}]\circ e_{12},$$
we deduce that $0=-2$, which is a contradiction. So $\al_1\ne 2$.
Similarly, we show that $\be_1\ne -2$.
\end{proof}

\begin{thm}\label{sl2}
There exists exactly one compatible root graded anti-pre-Lie
algebraic structure on ${\rm sl_2(\C)}$, which is the one given in
Example~{\rm \ref{ex:LB}}.
\end{thm}
\begin{proof}
By Examples \ref{ex:LB} and \ref{graded-on-sl2}, there exists a
compatible root graded anti-pre-Lie algebraic structure on ${\rm
sl_2(\C)}$. We still need to prove the uniqueness. In fact, it is
sufficient to show that
$$\al_1=-2,\ \be_1=2,\  \gamma_1=\frac{1}{2},\ \lambda_1=0,$$
in Eqs.~\eqref{sl2-1}-\eqref{sl2-4}.
By definition of anti-pre-Lie algebras we have
\begin{align*}
   e_{12}\circ(h_1\circ h_1)-h_1\circ(e_{12}\circ h_1)=[h_1, e_{12}]\circ h_1,\\
   e_{21}\circ(h_1\circ h_1)-h_1\circ(e_{21}\circ h_1)=[h_1, e_{21}]\circ h_1,
\end{align*}
where $[h_1, e_{12}]=h_1\circ e_{12}-e_{12}\circ h_1=2e_{12}, [h_1, e_{21}]=h_1\circ e_{21}-e_{21}\circ h_1=-2e_{21}$.
So
\begin{align}\label{xhh-yhh-sl2}
  (\lambda_1-\al_1)(\al_1-2)=2(\al_1-2), \ \  (\lambda_1-\be_1)(\be_1+2)=-2(\be_1+2).
\end{align}
By Lemma \ref{lemma-sl2} we obtain
\begin{align}\label{al1be1-sl2}
 \al_1=\lambda_1-2, \ \ \ \be_1=\lambda_1+2.
\end{align}
By the following equations
\begin{align*}
     &e_{21}\circ(e_{12}\circ e_{12})-e_{12}\circ(e_{21}\circ e_{12})=[e_{12}, e_{21}]\circ e_{12},\\
     &e_{12}\circ(e_{21}\circ e_{21})-e_{21}\circ(e_{12}\circ e_{21})=[e_{21}, e_{12}]\circ e_{21},
\end{align*}
we deduce
\begin{align}\label{yxx-xyy-sl2}
    -(\al_1-2)(\gamma_1-1)=\al_1\ \ {\rm and\ \ } -(\be_1+2)\gamma_1=-\be_1,
\end{align}
respectively.
Thus we have
\begin{align*}
    (\lambda_1-4)(\gamma_1-1)=2-\lambda_1 \ \ {\rm and \ \ } (\lambda_1+4)\gamma_1=\lambda_1+2
\end{align*}
by Eqs.~\eqref{al1be1-sl2} and \eqref{yxx-xyy-sl2},
which imply
$$\lambda_1(2\gamma_1-1)=0 \ \ {\rm and \ \ } 4(2\gamma_1-1)=\lambda_1.$$
So
$$\al_1=-2,\ \be_1=2,\ \gamma_1=\frac{1}{2},\ \lambda_1=0.$$
This completes the proof.
\end{proof}

\section{Compatible root graded anti-pre-Lie algebraic structures on any finite-dimensional complex simple Lie algebra (except ${\rm sl_2(\C)}$) }
 We firstly construct an $(n+2)$-dimensional Lie algebra $\bb_n$ for any $n\in\Z$ with $n\geq 2$.
Then we investigate the compatible anti-pre-Lie algebraic
structures, satisfying certain conditions, on $\bb_n$ by the
representation theory of ${\rm sl_2(\C)}$. Finally, we prove that
there does not exist a compatible root graded anti-pre-Lie
algebraic structure on any finite-dimensional complex simple Lie
algebra except ${\rm sl_2(\C)}$.

For any $n\in\Z$ with $n\geq 2$, let $\bb_n=\C x\oplus\C
y\oplus\oplus_{i=1}^n\C z_i$ be an $(n+2)$-dimensional Lie algebra
satisfying the following Lie brackets.
\begin{align*}
&[z_1, x]=2x,\ [z_1, y]=-2y,\ [x, y]=z_1,\ [z_2, x]=-x, \ [z_2, y]=y,\\
&[z_1, z_2]=0,\ [z_i, \bb_n]=0, \  3\leq i\leq n.
\end{align*}
It is clear that the subalgebra $\bb$ of $\bb_n$ spanned by $x, y, z_1$ is isomorphic to the simple Lie algebra ${\rm sl_2(\C)}$ under the following map.
$$x\mapsto e_{12},\  y\mapsto e_{21},\ z_1\mapsto h_1.$$

Now suppose that $(\bb_n,\circ)$ is a compatible anti-pre-Lie
algebraic structure on $\bb_n$ satisfying the following
conditions.
\begin{align}
&x\circ y\in\oplus_{k=1}^n\C z_k,\ \ \ z_i\circ x\in\C x,\ \ \
z_i\circ y\in \C y,   \label{graded-bn-1}\\
& z_i\circ z_j\in\oplus_{k=1}^n\C z_k, \ \
\  1\leq i, j\leq n,\ \ \ x\circ x=y\circ y=0.   \label{graded-bn-2}
\end{align}
Then by Eqs.~\eqref{graded-bn-1}, \eqref{graded-bn-2} and the Lie brackets of $\bb_n$, we can write that
 \begin{align}
 & z_1\circ x=\al_1 x,\ x\circ z_1=(\al_1-2)x,\
    z_1\circ y=\be_1 y,\ y\circ z_1=(\be_1+2)y,  \label{compatible-bn-1} \\
    & z_2\circ x=\al_2 x,\ x\circ z_2=(\al_2+1)x,\
    z_2\circ y=\be_2 y,\ y\circ z_2=(\be_2-1)y,  \label{compatible-bn-2}\\
    &z_t\circ x=x\circ z_t=\al_t x,\ z_t\circ y=y\circ z_t=\be_t y,  3\leq t\leq n, \label{compatible-bn-3}\\
    &x\circ y=\sum_{1\leq l\leq n}\gamma_lz_l,\ y\circ x=(\gamma_1-1)z_1+\sum_{2\leq l\leq n}\gamma_lz_l,  \label{compatible-bn-4}\\
    &z_p\circ z_q=z_q\circ z_p=\sum_{1\leq k\leq n}\lambda_{pq}^kz_k, \  1\leq p,q\leq n, \label{compatible-bn-5}
\end{align}
where $\al_i,\be_i,\gamma_i,\lambda_{pq}^k\in\C$, $1\leq
i,p,q,k\leq n$. Note that in addition,
\begin{align}\label{compatible-bn-6}
x\circ x=y\circ y=0. 
\end{align}
Define a linear map $\rho_n:\bb_n\rightarrow
\mathfrak{g}\mathfrak{l}(\bb_n)$ by
$$z\mapsto -L_z, \ \ \  z\in \bb_n,$$
where $-L_z$ denotes the negative left multiplication operator of
$(\bb_n,\circ)$, that is, $$-L_z(z')=-z\circ z', \ \ \
z'\in\bb_n.$$ Then $\rho_n$ defines a representation of the Lie
algebra $\bb_n$ by Lemma \ref{property-1}. Furthermore, $\rho_n$
also defines a representation of $\bb$ since $\bb$ is a subalgebra
of $\bb_n$.

In the remaining parts we identify $\bb$ with ${\rm sl_2(\C)}$
with $$x=e_{12},  y=e_{21}, z_1=h_1.$$ By Lemma
\ref{representation-2}, $\bb_n$ is viewed as a direct sum of the
certain finite-dimensional irreducible representations $\V(m)$ of
$\bb$, where $m\in\N$.

\begin{lem}\label{mgeq3}
With the above assumptions and notations,
for any $m\geq 3$, $\V(m)$ is not a
$\bb$-subrepresentation of $\bb_n$.
\end{lem}
\begin{proof}
    Assume that $\V(m)$ is a $\bb$-subrepresentation of $\bb_n$ with $m\geq 3$.
    Then by Lemma \ref{representation-1}, there exists a highest weight vector $(0\ne)\xi\in\V(m)\subseteq\bb_n$ of weight $m$ and  $y.(y.(y. \xi))=-y\circ(y\circ(y\circ\xi))$ is a nonzero weight vector of weight $m-6$.
    Denote $$\xi=\lambda_1x+\lambda_2y+\sum_{1\leq k\leq n}\lambda'_kz_k,$$
    where $\lambda_1,\lambda_2,\lambda_k'\in\C$ for $1\leq k\leq n$.
   It is straightforward to see that $-y\circ(y\circ(y\circ\xi))=0$ by Eqs.~\eqref{compatible-bn-1}-\eqref{compatible-bn-6}, which yields a contradiction.
   Therefore $\V(m)$ is not a $\bb$-subrepresentation of $\bb_n$ for any $m\geq 3$.
\end{proof}

By Eqs.~\eqref{compatible-bn-1} and \eqref{compatible-bn-6} we know that $x$ is a highest weight vector of weight $-\al_1$ when we regard $\bb_n$ as a representation of $\bb$.
By Lemmas \ref{representation-1} and \ref{mgeq3} we get
\begin{align}\label{012}
-\al_1\in\{0,1,2\}.
\end{align}
Similarly, we know that $y$ is a lowest weight vector of weight $-\be_1$ and
\begin{align}\label{0-1-2}
-\be_1\in\{0,-1,-2\}.
\end{align}

\begin{lem}\label{al01be2}
With the above assumptions and notations, the following
conclusions hold.
\begin{enumerate}
\item If $\al_1=0$ or $-1$, then $\be_1\ne 2$.
\label{it:a1}\item\label{it:a2} $\al_1\ne -1$. \item $\al_1\ne
0$.\label{it:a3}
\end{enumerate}
\end{lem}
\begin{proof}
(\ref{it:a1}). Assume that $\be_1=2$.  Then $y$ is a lowest weight
vector of weight $-2$. By Lemma \ref{representation-1} we get that
$(0\ne)x.(x.y)=x\circ(x\circ y)$ is a highest weight vector of
weight $2$. Thus we get that $x$ is a highest weight vector of
weight $2$ by
Eqs.~\eqref{compatible-bn-1}-\eqref{compatible-bn-4}, which
implies that $\al_1=-2$, which is a contradiction.

 (\ref{it:a2}).   Assume that $\al_1=-1$. Then $x$ is a highest weight vector of weight $1$.
Thus $\C x\oplus\C (y\circ x)$ is isomorphic to $\V(1)$ as
$\bb$-representations. Suppose that $f:\C x\oplus\C (y\circ
x)\rightarrow\V(1)$ is a $\bb$-representation isomorphism. Then we
 assume that $f(x)=\lambda_0 v_0$ by Lemma \ref{representation-1},
where $\lambda_0\in\C^*$. Hence we have
\begin{align*}
    \lambda_0v_0=x.(y.f(x))&=f(x\circ(y\circ x))\\
    &=((\gamma_1-1)(\al_1-2)+\gamma_2(\al_2+1)+\sum_{3\leq l\leq n}\gamma_l\al_l)\lambda_0v_0,
\end{align*}
which implies
\begin{align}\label{xyx}
     (\gamma_1-1)(\al_1-2)+\gamma_2(\al_2+1)+\sum_{3\leq l\leq n}\gamma_l\al_l=1.
\end{align}
Moreover, by Item~(\ref{it:a1}) we show that $y$ is a lowest
weight vector of weight $0$ or $-1$. Thus by Lemma
\ref{representation-1} again, we have
 \begin{align*}
     0=x.(x.y)=x\circ(x\circ y)=(\gamma_1(\al_1-2)+\gamma_2(\al_2+1)+\sum_{3\leq l\leq n}\gamma_l\al_l)x,
 \end{align*}
which yields
\begin{align}\label{xxy}
    \gamma_1(\al_1-2)+\gamma_2(\al_2+1)+\sum_{3\leq l\leq n}\gamma_l\al_l=0.
\end{align}
By Eqs.~\eqref{xyx}, \eqref{xxy} and the assumption that
$\al_1=-1$, we have $3=1$, which is a contradiction.

(\ref{it:a3}). Assume that $\al_1=0$. Then $x$ is a highest weight
vector of weight $0$. So $\C x$ is a trivial representation of
$\bb$ by Lemma \ref{representation-1}. By Item~(\ref{it:a1}),
$\be_1\ne 2$. Therefore we have
$$0=-y.x=y\circ x=(\gamma_1-1)z_1+\sum_{2\leq l\leq n}\gamma_l
z_l.$$ Hence $\gamma_1=1, \gamma_l=0, 2\leq l\leq n$. On the other
hand, we have $$0=x.(x. y)=x\circ(x\circ
y)=(\gamma_1(\al_1-2)+\gamma_2(\al_2+1)+\sum_{3\leq l\leq
n}\gamma_l\al_l)x=-2x,$$
 which is a contradiction.
\end{proof}

\begin{cor}\label{al1=-2}
With the above assumptions and notations, $\al_1=-2$ and $x$
is a highest weight vector of weight $2$ when we view $\bb_n$ as a
representation of $\bb$. Similarly, $\be_1=2$ and $y$ is a lowest
weight vector of weight $-2$  when we view $\bb_n$ as a
representation of $\bb$.
\end{cor}

\begin{proof}
It follows from Lemma~\ref{al01be2} (\ref{it:a2}), (\ref{it:a3})
and Eqs.~\eqref{012}, \eqref{0-1-2}.
\end{proof}

\begin{lem}\label{weight2}
Let $\bb_n$ be viewed as a representation of $\bb$. Then the
 following  conclusions hold.
\begin{enumerate}
\item \label{it:b1} $x$ is the unique highest weight vector of
weight $2$ up to a nonzero scalar
 and $y$ is the unique lowest weight vector of weight $-2$ up to a nonzero scalar.
\item \label{it:b2} There does not exist a highest weight vector
of weight $1$.
\end{enumerate}
\end{lem}
\begin{proof}
(\ref{it:b1}). Assume that $x'=\sum_{1\leq l\leq n}\mu_lz_l$ is a
highest weight vector of weight $2$. Then we show that $y\circ x'$
is a nonzero weight vector of weight $0$  by Lemma
\ref{representation-1}. Thus $y$ is a weight vector of weight $0$
by Eqs.~\eqref{compatible-bn-1}-\eqref{compatible-bn-3}, which
contradicts Corollary \ref{al1=-2}. Hence $x$ is the unique
highest weight vector of weight $2$ up to a nonzero scalar.
Similarly,  we show that $y$ is the unique lowest weight vector of
weight $-2$ up to a nonzero scalar.

(\ref{it:b2}). Assume that $x'=\sum_{1\leq l\leq n}\mu_lz_l$ is a
nonzero highest weight vector of weight $1$. Then we deduce that $y\circ
x'$ is a nonzero weight vector of weight $-1$ by Lemma
\ref{representation-1}. So $y$ is a weight vector of weight $-1$
by Eqs.~\eqref{compatible-bn-1}-\eqref{compatible-bn-3}, which
contradicts Corollary \ref{al1=-2}.
\end{proof}

\begin{pro}\label{V2(n-1)V0}
Suppose that $(\bb_n, \circ)$ is a compatible anti-pre-Lie
algebraic structure on $\bb_n$ satisfying Eqs.~\eqref{graded-bn-1} and \eqref{graded-bn-2}.
Then the following equations hold.
\begin{align}
   & z_1\circ x=-2 x,\ x\circ z_1=-4x,\
    z_1\circ y=2 y,\ y\circ z_1=4y,   \label{V2(n-1)V0-1} \\
    & z_2\circ x=x,\ x\circ z_2=2x,\
    z_2\circ y=-y,\ y\circ z_2=-2y, \label{V2(n-1)V0-2} \\
    &x\circ y=\frac{1}{2}z_1,\ y\circ x=-\frac{1}{2}z_1,  \label{V2(n-1)V0-3} \\
    &z_1\circ z_i=z_i\circ z_1=0, \ \ 1\leq i\leq n, \label{V2(n-1)V0-4}\\
     &z_j\circ x=x\circ z_j=z_j\circ y=y\circ z_j=0,\ \ 3\leq j\leq n. \label{V2(n-1)V0-5}
\end{align}
\end{pro}
\begin{proof}
By assumption,
Eqs.~\eqref{compatible-bn-1}-\eqref{compatible-bn-6} hold. Thus by
Corollary \ref{al1=-2}, we have $\al_1=-2,\be_1=2$. So
Eq.~\eqref{V2(n-1)V0-1} holds.
 Moreover, by Lemmas \ref{representation-1}, \ref{representation-2}, \ref{mgeq3} and \ref{weight2}, we deduce that  $\bb_n$ is isomorphic to $\V(2)\oplus \oplus_{1\leq d\leq n-1}\V^d(0)$ as
 representations of $\bb$, where $\V^d(0)=\C v^d_0$ is the one-dimensional trivial representation of $\bb$, $1\leq d\leq n-1$.
    Let $$f:\bb_n\rightarrow \V(2)\oplus \oplus_{1\leq d\leq n-1}\V^d(0),$$ be the $\bb$-representation isomorphism.
    By Lemma \ref{weight2}, without loss of generality, we assume that $$f(x)=v_0\in\V(2),\ \ \ f(y)=\lambda v_2\in\V(2),$$ where $\lambda\in\C^*$.
    Then we have
    \begin{align}
        &v_1=y.v_0=y.f(x)=-f(y\circ x)=-f((\gamma_1-1)z_1+\sum_{2\leq l\leq n}\gamma_lz_l),\label{yv0}\\
        &\lambda v_1=\lambda x.v_2=x.f(y)=-f(x\circ y)=-f(\gamma_1z_1+\sum_{2\leq l\leq n}\gamma_lz_l).\label{xv2}
    \end{align}
    So $f(z_1)=(1-\lambda)v_1$ which yields $\lambda\ne 1$ since $f$ is an isomorphism.
    Thus we get
    \begin{align*}
        2(1-\lambda)v_2=(1-\lambda)y.v_1=y.f(z_1)=-f(y\circ z_1)=-\lambda(\be_1+2)v_2.
    \end{align*}
    Therefore $\lambda=-1$ since $\be_1=2$ and $f(z_1)=2v_1, f(y)=-v_2$.
    Then by Eq.~\eqref{yv0} we get
    \begin{align*}
       f(z_1)=2v_1=-2f((\gamma_1-1)z_1+\sum_{2\leq l\leq n}\gamma_lz_l),
    \end{align*}
    which yields $$z_1=-2(\gamma_1-1)z_1-2\sum_{2\leq l\leq n}\gamma_lz_l,$$
    since $f$ is an isomorphism.
    Thus $\gamma_1=\frac{1}{2}, \gamma_l=0$, $2\leq l\leq n$. So Eq.~\eqref{V2(n-1)V0-3} holds.
 Furthermore,  we have $$0=2z_1.v_1=z_1.f(z_1)=-f(z_1\circ z_1),$$
    which implies
    \begin{align}\label{z1z1}
    z_1\circ z_1=0,
    \end{align}
     since $f$ is an isomorphism.
Assume that for any $2\leq j\leq n$,
$$f(z_j)=\lambda_{j0}v_0+\lambda_{j1}v_1+\lambda_{j2}v_2+\sum_{1\leq d\leq
n-1}\mu_{jd}v_0^d,$$
    where $\lambda_{j0},\lambda_{j1},\lambda_{j2},\mu_{jd}\in\C$, $1\leq d\leq n-1$.
    Then we have
    \begin{align*}
        &2\lambda_{21}v_0+\lambda_{22}v_1=x.f(z_2)=-f(x\circ z_2)=-(\al_2+1)v_0,\\
        &\lambda_{20}v_1+2\lambda_{21}v_2=y.f(z_2)=-f(y\circ z_2)=(\be_2-1)v_2,\\
          &2\lambda_{j1}v_0+\lambda_{j2}v_1=x.f(z_j)=-f(x\circ z_j)=-\al_jv_0,\ \ \ 3\leq j\leq n,\\
        &\lambda_{j0}v_1+2\lambda_{j1}v_2=y.f(z_j)=-f(y\circ z_j)=\be_jv_2, \ \ \ 3\leq j\leq n.
    \end{align*}
Thus
\begin{align*}
&\lambda_{20}=\lambda_{22}=0,  \  2\lambda_{21}=-(\al_2+1)=(\be_2-1),  \\
&\lambda_{j0}=\lambda_{j2}=0,  \  2\lambda_{j1}=-\al_j=\be_j, \  3\leq j\leq n. 
\end{align*}
  So
$$f(z_j)=\lambda_{j1}v_1+\sum_{1\leq d\leq n-1}\mu_{jd}v_0^d,\ \ \ 2\leq j\leq n.$$
    Then
  $$ 0=z_1.f(z_j)=-f(z_1\circ z_j),$$
    which implies
        \begin{align}\label{z1z2}
    z_1\circ z_j=0, \ \ \ 2\leq j\leq n,
    \end{align} since $f$ is an isomorphism.
   By Eqs.~\eqref{z1z1} and \eqref{z1z2}, Eq.~\eqref{V2(n-1)V0-4} holds.
Finally, we have
    \begin{align*}
        &0=z_2\circ(x\circ z_1)-x\circ(z_2\circ z_1)-[x, z_2]\circ z_1
        =-4z_2\circ x+4x=(-4\al_2+4)x,\\
        &0=z_2\circ(y\circ z_1)-y\circ(z_2\circ z_1)-[y, z_2]\circ z_1
        =4z_2\circ y+4y=(4\be_2+4)y,\\
        &0=z_j\circ(x\circ z_1)-x\circ(z_j\circ z_1)-[x, z_j]\circ z_1
        =-4z_j\circ x=-4\al_jx,\ \ \ 3\leq j\leq n,\\
        &0=z_j\circ(y\circ z_1)-y\circ(z_j\circ z_1)-[y, z_j]\circ z_1
        =4z_j\circ y=4\be_j y, \ \ \ 3\leq j\leq n.
    \end{align*}
  So $\al_2=1, \be_2=-1, \al_j=\be_j=0$ for $3\leq j\leq n$. Thus Eqs.~\eqref{V2(n-1)V0-2} and \eqref{V2(n-1)V0-5} hold.
   This completes the proof of the conclusion.
\end{proof}

\begin{thm}\label{main}
Let $n\in\Z$ with $n\geq 2$ and $\gg$ be a Lie algebra with an
$n$-dimensional abelian subalgebra $\hh=\oplus_{i=1}^n\C z'_i$.
Suppose that there exist $x_1, y_1, x_2, y_2, x_3, y_3\in\gg$ such
that the following conditions hold.
\begin{enumerate}
\item $\gg_1=\C x_1\oplus \C y_1\oplus\hh$ is a subalgebra of
$\gg$ and $\gg_1$ is isomorphic to $\bb_n$ as Lie algebras.
Moreover, there exists a Lie algebra isomorphism $T_1: \gg_1\rightarrow \bb_n$
such that
$$T_1(x_1)= x,\ \ \ T_1(y_1)= y, \ \ \ T_1(z'_1)= z_1, \ \ \ T_1(z'_2)=z_2.$$
\item $\gg_2=\C x_2\oplus \C y_2\oplus\hh$ is a subalgebra of
$\gg$ and $\gg_2$ is isomorphic to $\bb_n$ as Lie algebras.
Moreover, there exists a Lie algebra isomorphism $T_2: \gg_2\rightarrow \bb_n$
such that
$$T_2(x_2)=x,\ \ \ T_2(y_2)=y, \ \ \ T_2(z'_2)=z_1, \ \ \ T_2(z'_1)=z_2.$$
 \item $\gg_3=\C x_3\oplus \C y_3\oplus\hh$
is a subalgebra of $\gg$ and $\gg_3$ is isomorphic to $\bb_n$ as Lie algebras.
Moreover, there exists a Lie algebra isomorphism $T_3: \gg_3\rightarrow \bb_n$
such that
$$T_3(x_3)=x,\ \ \ T_3(y_3)=y, \ \ \ T_3(z'_1+z'_2)= z_1.$$
\item $[x_1, x_2]=x_3$.
\end{enumerate}
Then there does not exist a compatible anti-pre-Lie algebraic
structure $(\gg, \circ)$ on $\gg$ satisfying the following
conditions.
\begin{align}
&x_k\circ y_k\in \hh,\ \ \   z'_i\circ x_k\in\C x_k,\ \ \ z'_i\circ
y_k\in \C y_k,  \label{graded-gg-1}\\
& z'_i\circ z'_j\in\hh,\ \ \  x_k\circ x_k=y_k\circ y_k=0, \ \ \   1\leq i, j\leq
n, 1\leq k\leq 3.\label{graded-gg-2}
\end{align}
\end{thm}
\begin{proof}
Assume that $(\gg, \circ)$ is a compatible anti-pre-Lie algebraic
structure on $\gg$ satisfying Eq.~\eqref{graded-gg-1} and
\eqref{graded-gg-2}. Then $(\gg_1, \circ), (\gg_2, \circ)$ and
$(\gg_3, \circ)$ are compatible anti-pre-Lie algebraic structures
on $\gg_1, \gg_2$ and $\gg_3$ respectively. Note that
Eqs.~\eqref{graded-gg-1} and \eqref{graded-gg-2}  are exactly
Eqs.~\eqref{graded-bn-1} and \eqref{graded-bn-2} respectively in
each case. Therefore by Proposition \ref{V2(n-1)V0}, we have
$$x_1\circ z'_1=-4x_1,\  x_1\circ z'_2=2x_1,\ x_2\circ z'_2=-4x_2,\  x_2\circ z'_1=2x_2,\ x_3\circ (z'_1+z'_2)=-4x_3.$$
Hence we obtain
\begin{align*}
-4x_3&=x_3\circ (z'_1+z'_2)=[x_1, x_2]\circ (z'_1+z'_2)\\
&=x_2\circ(x_1\circ (z'_1+z'_2))-x_1\circ(x_2\circ (z'_1+z'_2))=-2x_2\circ x_1+2x_1\circ x_2 \\
&=2x_3,
\end{align*}
which is a contradiction. This completes the proof of the
conclusion.
\end{proof}

Next we apply Theorem~\ref{main} to study the compatible root
graded anti-pre-Lie algebraic structures on the following
finite-dimensional complex simple Lie algebras
$$A_n (n\geq 2), B_n (n\geq 2), C_n(n\geq 3), D_n (n\geq 4), E_6, E_7, E_8, F_4, G_2,$$ case by case.
For any $m\in\Z_+$, let ${\rm gl_m(\C)}$ denote the set of all
$m\times m$ matrices over $\C$. Moreover,  $e_{ij}$ ($1\leq i,
j\leq m$) denotes the matrix whose $i$-th row and $j$-th column is
1 and other positions are zero.

\begin{case}
The simple Lie algebra ${\rm sl_{n+1}(\C)}$, which is $A_n$,
$n\geq 2$.
\end{case}
 The simple Lie algebra ${\rm sl_{n+1}(\C)}$  $(n\geq 2)$ consists of all $(n+1)\times (n+1)$ matrices of trace zero (\cite{Car}).
 When $n\geq 3$, ${\rm sl_{n+1}(\C)}$ has the following Lie subalgebras.
 \begin{align*}
& AS_n=\C e_{12}\oplus \C e_{21}\oplus\oplus_{1\leq k\leq n}\C h_k, \\
&AS'_n=\C e_{23}\oplus \C e_{32}\oplus\C h_2\oplus\C h_1\oplus\C (e_{11}-e_{44})\oplus\oplus_{4\leq k\leq n}\C h_k,\\
&AS''_n=\C e_{13}\oplus \C e_{31}\oplus \C (e_{11}-e_{33})\oplus\C (e_{33}-e_{22})\oplus\C (e_{22}-e_{44})\oplus\oplus_{4\leq k\leq n}\C h_k,
 \end{align*}
 where $h_k=e_{kk}-e_{(k+1)(k+1)}$, $1\leq k\leq n$.
By direct observations, $AS_n, AS'_n$ and $AS''_n$ are isomorphic to $\bb_n$ under the following correspondences respectively.
\begin{align*}
&e_{12}\rightarrow x,\  e_{21}\rightarrow y,\  h_k\rightarrow z_k, \   1\leq k\leq n,\\
&e_{23}\rightarrow x,\  e_{32}\rightarrow y,\ h_2\rightarrow z_1,\  h_1\rightarrow z_2,\ e_{11}-e_{44}\rightarrow z_3,\ h_k\rightarrow z_k, \   4\leq k\leq n,\\
&e_{13}\rightarrow x,\  e_{31}\rightarrow y,\
e_{11}-e_{33}\rightarrow z_1,\  e_{33}-e_{22}\rightarrow z_2,\
e_{22}-e_{44}\rightarrow z_3,\\
& h_k\rightarrow z_k, \   4\leq k\leq n.
\end{align*}
Note that the fact that there is a compatible root graded
anti-pre-Lie algebraic structure on ${\rm sl_{n+1}(\C)}$ implies
that Eqs.~\eqref{graded-gg-1} and \eqref{graded-gg-2} hold. Therefore by Theorem \ref{main}
we get the following conclusion.
\begin{thm}\label{An}
    There does not exist a compatible root graded anti-pre-Lie algebraic structure on ${\rm sl_{n+1}(\C)}$ for any $n\geq 3$.
\end{thm}

When $n=2$, ${\rm sl_3(\C)}$ has the following Lie subalgebras.
  \begin{align}\label{AS2}
& AS_2=\C e_{12}\oplus \C e_{21}\oplus\C h_1\oplus\C h_2, \\
&AS'_2=\C e_{23}\oplus \C e_{32}\oplus\C h_2\oplus\C h_1,\\
&AS''_2=\C e_{13}\oplus \C e_{31}\oplus \C (e_{11}-e_{33})\oplus\C (e_{33}-e_{22}),
 \end{align}
 where $h_1=e_{11}-e_{22}, h_2=e_{22}-e_{33}$. Obviously, $AS_2, AS'_2$ and $AS''_2$ are isomorphic to
 $\bb_2$.
 Then by Theorem \ref{main} again we obtain the following result.
\begin{thm}\label{A2}
    There does not exist a compatible root graded anti-pre-Lie algebraic structure on ${\rm sl_3(\C)}$.
\end{thm}

\begin{case}
The simple Lie algebra $\oo(2n+1, \C)$, which is $B_n$, $n\geq 2$.
 \end{case}
The simple Lie algebra $\oo(2n+1, \C)$ is a Lie subalgebra of
${\rm gl_{2n+1}(\C)}$. For any matrix in $\oo(2n+1, \C)$, we
renumber its rows and columns as $0,1,\cdots, n,n+1,\cdots, 2n$.
When $n\geq 3$, $\oo(2n+1, \C)$ has the following Lie subalgebras
(\cite{Car}).
\begin{align}
BS_n={\rm span}\{&e_{12}-e_{(n+2)(n+1)}, e_{21}-e_{(n+1)(n+2)},  \notag \\
 &h_1-h_2, h_2-h_3, h_i~|~3\leq i\leq n\}, \label{BSn}\\
BS'_n={\rm span}\{&e_{23}-e_{(n+3)(n+2)}, e_{32}-e_{(n+2)(n+3)}, \notag \\
&h_2-h_3, h_1-h_2, h_1, h_i~|~4\leq i\leq n\},   \label{BSn'}\\
BS''_n={\rm span}\{&e_{13}-e_{(n+3)(n+1)}, e_{31}-e_{(n+1)(n+3)},  \notag  \\
&h_1-h_3, h_3, h_2, h_i~|~4\leq i\leq n\},  \label{BSn''}
\end{align}
where $h_i=e_{ii}-e_{(n+i)(n+i)}$,  $1\leq i\leq n$. It is clear
that $BS_n, BS'_n$ and $BS''_n$ are isomorphic to $\bb_n$ under
the following correspondences respectively.
\begin{align*}
&e_{12}-e_{(n+2)(n+1)}\rightarrow x,\ \ \  e_{21}-e_{(n+1)(n+2)}\rightarrow y,\\
& h_1-h_2 \rightarrow z_1,\ \  h_2-h_3\rightarrow z_2,\ \    h_i\rightarrow z_i, \   \    3\leq i\leq n,\\
&e_{23}-e_{(n+3)(n+2)}\rightarrow x,\  \ \ e_{32}-e_{(n+2)(n+3)}\rightarrow y,\  \ \ h_2-h_3\rightarrow z_1,\\
&h_1-h_2\rightarrow z_2,\   \ \ h_1\rightarrow z_3, \ \ \  h_i\rightarrow z_i, \ \ \   4\leq i\leq n,\\
&e_{13}-e_{(n+3)(n+1)}\rightarrow x,\ \ \  e_{31}-e_{(n+1)(n+3)}\rightarrow y,\ \ \ h_1-h_3\rightarrow z_1,\\
&h_3\rightarrow z_2,\ \ \ h_2\rightarrow z_3,  \ \ \
h_i\rightarrow z_i, \  \ \   4\leq i\leq n.
\end{align*}
Then by Theorem \ref{main} we get the following result.
\begin{thm}\label{Bn}
    There does not exist a compatible root graded anti-pre-Lie algebraic structure on $\oo(2n+1, \C)$ for any $n\geq 3$.
\end{thm}

When $n=2$, $\oo(5, \C)$ still has three Lie subalgebras which are
isomorphic to $\bb_2$. For example, set
\begin{align*}
&BS_2={\rm span}\{2e_{10}-e_{03}, e_{01}-2e_{30}, 2(e_{11}-e_{33}), -e_{11}+e_{22}+e_{33}-e_{44}\},  \\
&BS'_2={\rm span}\{2e_{20}-e_{04}, e_{02}-2e_{40}, 2(e_{22}-e_{44}), e_{11}-e_{22}-e_{33}+e_{44}\},  \\
&BS''_2={\rm span}\{e_{23}-e_{14}, e_{32}-e_{41},
e_{11}+e_{22}-e_{33}-e_{44}, -e_{11}+e_{33}\}.
\end{align*}
It is straightforward to show that $BS_2, BS'_2$ and $BS''_2$ are
isomorphic to $\bb_2$ under the following correspondences respectively. 
\begin{align*}
&2e_{10}-e_{03}\rightarrow x,\  e_{01}-2e_{30}\rightarrow y,\ 2(e_{11}-e_{33})\rightarrow z_1,\  -e_{11}+e_{22}+e_{33}-e_{44}\rightarrow z_2,\\
&2e_{20}-e_{04}\rightarrow x,\  e_{02}-2e_{40}\rightarrow y,\  2(e_{22}-e_{44})\rightarrow z_1,\ e_{11}-e_{22}-e_{33}+e_{44}\rightarrow z_2,\\
&e_{23}-e_{14}\rightarrow x,\  e_{32}-e_{41}\rightarrow y,\ 
e_{11}+e_{22}-e_{33}-e_{44}\rightarrow z_1,\ -e_{11}+e_{33}\rightarrow z_2.
\end{align*}
But we notice $$[2e_{10}-e_{03}, 2e_{20}-e_{04}]=2(e_{23}-e_{14}).$$ 
So $BS_2, BS'_2$ and $BS''_2$ do not satisfy the conditions of Theorem \ref{main}.
 However we
still have the following conclusion although the proof is
 similar to the one for Theorem \ref{main}.
\begin{thm}\label{B2}
    There does not exist a compatible root graded anti-pre-Lie algebraic structure on $\oo(5, \C)$.
\end{thm}
\begin{proof}
Assume that $(\oo(5, \C), \circ)$ is a compatible root graded
anti-pre-Lie algebraic structure on $\oo(5, \C)$. Then by the root
graded condition and Proposition \ref{V2(n-1)V0} we get
\begin{align*}
&(2e_{10}-e_{03})\circ 2(e_{11}-e_{33})=-4(2e_{10}-e_{03}),\\
&(2e_{10}-e_{03})\circ (-e_{11}+e_{22}+e_{33}-e_{44})=2(2e_{10}-e_{03}),\\
&(2e_{20}-e_{04})\circ 2(e_{22}-e_{44})=-4(2e_{20}-e_{04}),  \\
&(2e_{20}-e_{04})\circ (e_{11}-e_{22}-e_{33}+e_{44})=2(2e_{20}-e_{04}),  \\
&(e_{23}-e_{14})\circ (e_{11}+e_{22}-e_{33}-e_{44})=-4(e_{23}-e_{14}).
\end{align*}
Hence we have
\begin{align*}
-8(e_{23}-e_{14})=&2(e_{23}-e_{14})\circ (e_{11}+e_{22}-e_{33}-e_{44})\\
=&[2e_{10}-e_{03}, 2e_{20}-e_{04}]\circ(e_{11}+e_{22}-e_{33}-e_{44})\\
=&(2e_{20}-e_{04})\circ ((2e_{10}-e_{03})\circ (e_{11}+e_{22}-e_{33}-e_{44})) \\
&-(2e_{10}-e_{03})\circ ((2e_{20}-e_{04})\circ(e_{11}+e_{22}-e_{33}-e_{44}))\\
=&-2(2e_{20}-e_{04})\circ (2e_{10}-e_{03})+2(2e_{10}-e_{03})\circ (2e_{20}-e_{04})\\
=&4(e_{23}-e_{14}),
\end{align*}
which is a contradiction. This completes the proof of the
conclusion.
\end{proof}

\begin{case}
The simple Lie algebra $\sp(2n, \C)$, which is $C_n$, $n\geq 3$.
 \end{case}
The simple Lie algebra $\sp(2n, \C)$ is a Lie subalgebra of  ${\rm
gl_{2n}(\C)}$. For any matrix in $\sp(2n, \C)$, we still number
its rows and columns as $1,2, \cdots, n, n+1,\cdots, 2n$. For any
$n\geq 3$, $BS_n, BS'_n$ and $BS''_n$, defined by
Eqs.~\eqref{BSn}, \eqref{BSn'} and  \eqref{BSn''} respectively,
are the Lie subalgebras of $\sp(2n, \C)$  (\cite{Car}). Then by
Theorem \ref{main} we get the following result.
\begin{thm}\label{Cn}
    There does not exist a compatible root graded anti-pre-Lie algebraic structure on $\sp(2n, \C)$ for any $n\geq 3$.
\end{thm}

\begin{case}
The simple Lie algebra $\oo(2n, \C)$, which is $D_n$, $n\geq 4$.
\end{case}

The simple Lie algebra $\oo(2n, \C)$ is a Lie subalgebra of  ${\rm
gl_{2n}(\C)}$. For any matrix in $\oo(2n, \C)$, we still number
its rows and columns as $1,2, \cdots, n, n+1,\cdots, 2n$. For any
$n\geq 4$, $BS_n, BS'_n$ and $BS''_n$, defined by
Eqs.~\eqref{BSn}, \eqref{BSn'} and  \eqref{BSn''} respectively,
are the  Lie subalgebras of $\oo(2n, \C)$ (\cite{Car}). Then by
Theorem \ref{main} we get the following result.
\begin{thm}\label{Dn}
    There does not exist a compatible root graded anti-pre-Lie algebraic structure on $\oo(2n, \C)$ for any $n\geq 4$.
\end{thm}

\begin{case}
The simple Lie algebras $\ee_6, \ee_7$ and $\ee_8$, which are
$E_6$, $E_7$ and $E_8$ respectively.
\end{case}

Recall that (\cite{Car, Hum}) the Cartan matrix of $E_6$ is
\begin{align*}
M=\begin{pmatrix}
2&-1&0&0&0&0\\
-1&2&-1&0&0&0\\
0&-1&2&-1&-1&0\\
0&0&-1&2&0&0\\
0&0&-1&0&2&-1\\
0&0&0&0&-1&2
\end{pmatrix}.
\end{align*}
The simple Lie algebra $\ee_6$ is generated by $e_1, e_2, \cdots,
e_6, h'_1, h'_2, \cdots, h'_6, f_1, f_2, \cdots, f_6$ with the
following relations.
\begin{align*}
&[h'_i, h'_j]=[h'_i, e_j]-M_{ij}e_j=[h'_i, f_j]+M_{ij}f_j=[e_i, f_i]-h'_i=0, \ \ 1\leq i, j\leq 6, \\
&[e_i, f_j]=ad_{e_i}^{1-M_{ij}}e_j=ad_{f_i}^{1-M_{ij}}f_j=0, \ \ \
 1\leq i\ne j\leq 6,
\end{align*}
where $M_{ij}$ is the element in the $i$-th row and $j$-th column
of the matrix $M$, $1\leq i,j\leq 6$ and $ad_{e_i}e_j=[e_i, e_j],
ad_{f_i}f_j=[f_i, f_j]$, $1\leq i\ne j\leq 6$. Thus the Lie
algebra $\ee_6$ has the following Lie subalgebras.
\begin{align*}
&ES_6={\rm span}\{e_1, f_1, h'_1, h'_2, h'_3, h'_4, h'_5, h'_6\}, \\
&ES'_6={\rm span}\{-e_2, -f_2, h'_2, h'_1, h'_1-h'_3, h'_4, h'_5, h'_6\}, \\
&ES''_6={\rm span}\{-[e_1, e_2], [f_1, f_2], h'_1+h'_2, h'_3,
h'_1-h'_2, h'_4, h'_5, h'_6\},
\end{align*}
are isomorphic to $\bb_6$ under the following correspondences
respectively.
\begin{align*}
&e_1\mapsto x,\ f_1\mapsto y,\  h'_i\mapsto z_i, \  1\leq i\leq 6,\\
&-e_2\mapsto x,\ -f_2\mapsto y,\  h'_2\mapsto z_1,\ h'_1\mapsto z_2,\  h'_1-h'_3\mapsto z_3,\ h'_i\mapsto z_i, \ 4\leq i\leq 6,\\
&-[e_1, e_2]\mapsto x,  [f_1, f_2]\mapsto y,  h'_1+h'_2\mapsto
z_1,  h'_3\mapsto z_2, h'_1-h'_2\mapsto z_3, h'_i\mapsto z_i, 4\leq i\leq 6.
\end{align*}
Then by Theorem \ref{main} we get the following result.
\begin{thm}\label{E}
    There does not exist a compatible root graded anti-pre-Lie algebraic structure on $\ee_6$.
Similarly, there does not exist a compatible root graded
anti-pre-Lie algebraic structure on either $\ee_7$ or $ \ee_8$.
\end{thm}

\begin{case}
The simple Lie algebra $\ff_4$, which is $F_4$.
\end{case}
Recall that (\cite{Car, Hum}) the Cartan matrix of $F_4$ is
\begin{align*}
N=\begin{pmatrix}
2&-1&0&0\\
-1&2&-1&0\\
0&-2&2&-1\\
0&0&-1&2
\end{pmatrix}.
\end{align*}
The simple Lie algebra $\ff_4$ is generated by $e_1, e_2, e_3, e_4, h'_1, h'_2, h'_3, h'_4, f_1, f_2, f_3, f_4$ with the following relations.
\begin{align*}
&[h'_i, h'_j]=[h'_i, e_j]-N_{ij}e_j=[h'_i, f_j]+N_{ij}f_j=[e_i, f_i]-h'_i=0, \ \  1\leq i, j\leq 4, \\
&[e_i, f_j]=ad_{e_i}^{1-N_{ij}}e_j=ad_{f_i}^{1-N_{ij}}f_j=0, \ \  1\leq i\ne j\leq 4,
\end{align*}
where $N_{ij}$ is the element in the $i$-th row and $j$-th column
of the matrix $N$,  $1\leq i,j\leq 4$ and $ad_{e_i}e_j=[e_i, e_j],
ad_{f_i}f_j=[f_i, f_j]$, $1\leq i\ne j\leq 4$. Thus the Lie
algebra $\ff_4$ has the following Lie subalgebras.
\begin{align*}
&FS_4={\rm span}\{e_1, f_1, h'_1, h'_2, h'_3, h'_4\}, \\
&FS'_4={\rm span}\{-e_2, -f_2, h'_2, h'_1, 2h'_1-h'_3, h'_4\}, \\
&FS''_4={\rm span}\{-[e_1, e_2], [f_1, f_2], h'_1+h'_2, \frac{1}{2}h'_3, h'_1-h'_2, h'_4\},
\end{align*}
are isomorphic to $\bb_4$ under the following correspondences
respectively.
\begin{align*}
&e_1\mapsto x,\ f_1\mapsto y,\  h'_i\mapsto z_i, \ \  \ 1\leq i\leq 4,\\
&-e_2\mapsto x,\ -f_2\mapsto y,\  h'_2\mapsto z_1,\ h'_1\mapsto z_2,\  2h'_1-h'_3\mapsto z_3,\ h'_4\mapsto z_4, \\
&-[e_1, e_2]\mapsto x,\  [f_1, f_2]\mapsto y,\  h'_1+h'_2\mapsto z_1,\  \frac{1}{2}h'_3\mapsto z_2,\ h'_1-h'_2\mapsto z_3,\  h'_4\mapsto z_4.
\end{align*}
Then by Theorem \ref{main} we get the following result.
\begin{thm}
    There does not exist a compatible root graded anti-pre-Lie algebraic structure on $\ff_4$.
\end{thm}

\begin{case}
The simple Lie algebra $\gg_2$, which is $G_2$.
\end{case}

Recall that (\cite{Car, Hum}) the Cartan matrix of $G_2$ is
\begin{align*}
G=\begin{pmatrix}
2&-1\\
-3&2
\end{pmatrix}.
\end{align*}
The simple Lie algebra $\gg_2$ is generated by $e_1, e_2, h'_1, h'_2, f_1, f_2$ with the
following relations.
\begin{align*}
&[h'_i, h'_j]=[h'_i, e_j]-G_{ij}e_j=[h'_i, f_j]+G_{ij}f_j=[e_i, f_i]-h'_i=0, \ \  1\leq i, j\leq 2, \\
&[e_i, f_j]=ad_{e_i}^{1-G_{ij}}e_j=ad_{f_i}^{1-G_{ij}}f_j=0, \ \  1\leq i\ne j\leq 2,
\end{align*}
where $G_{ij}$ is the element in the $i$-th row and $j$-th column
of the matrix $G$,  $1\leq i,j\leq 2$ and $ad_{e_i}e_j=[e_i,
e_j]$, $ad_{f_i}f_j=[f_i, f_j]$, $1\leq i\ne j\leq 2$. Obviously
the Lie algebra $\gg_2$ still has three different Lie subalgebras
which are isomorphic to $\bb_2$. For example,
\begin{align*}
&GS_2={\rm span}\{e_1, f_1, h'_1, \frac{1}{3}h'_2\}, \\
&GS'_2={\rm span}\{e_2, f_2, h'_2, h'_1\}, \\
&GS''_2={\rm span}\{-[e_1, e_2], [f_1, f_2], 3h'_1+h'_2, h'_2\},
\end{align*}
are isomorphic to $\bb_2$ under the following correspondences
respectively.
\begin{align*}
&e_1\mapsto x,\ f_1\mapsto y,\  h'_1\mapsto z_1, \ \frac{1}{3}h'_2\rightarrow z_2, \\
&e_2\mapsto x,\ f_2\mapsto y,\  h'_2\mapsto z_1,\ h'_1\mapsto z_2,\\
&-[e_1, e_2]\mapsto x,\  [f_1, f_2]\mapsto y,\  3h'_1+h'_2\mapsto z_1,\  h'_2\mapsto z_2.
\end{align*}
But we cannot directly apply the Theorem \ref{main} since $GS_2,
GS'_2$ and $GS''_2$ do not satisfy the conditions of Theorem
\ref{main}. Nevertheless, we still have the following conclusion.

\begin{thm}\label{thm:g2}
    There does not exist a compatible root graded anti-pre-Lie algebraic structure on $\gg_2$.
\end{thm}
\begin{proof}
Assume that $(\gg_2, \circ)$ is a compatible root graded
anti-pre-Lie algebraic structure on $\gg_2$. Then by the root
graded condition and Proposition \ref{V2(n-1)V0} we get
\begin{align*}
&e_1\circ h'_1=-4e_1, \  e_1\circ \frac{1}{3}h'_2=2e_1,\ e_2\circ h'_2=-4e_2, \ e_2\circ h'_1=2e_2, \\
& -[e_1, e_2]\circ (3h'_1+h'_2)=4[e_1, e_2],\  -[e_1, e_2]\circ h'_2= -2[e_1, e_2].
\end{align*}
Hence we deduce
\begin{align*}
\frac{4}{5}[e_1, e_2]&= -[e_1, e_2]\circ (h'_1+\frac{3}{5}h'_2)\\
&=e_1\circ(e_2\circ (h'_1+\frac{3}{5}h'_2))-e_2\circ(e_1\circ (h'_1+\frac{3}{5}h'_2))=-\frac{2}{5}e_1\circ e_2+\frac{2}{5}e_2\circ e_1\\
&=-\frac{2}{5}[e_1, e_2],
\end{align*}
which is a contradiction. This completes the proof of the
conclusion.
\end{proof}

Combining Theorems~\ref{An}-\ref{thm:g2} together, we have the
following conclusion.

\begin{cor}
There does not exist a compatible root graded anti-pre-Lie
algebraic structure on any finite-dimensional complex simple Lie
algebra except ${\rm sl_2(\C)}$.
\end{cor}

\noindent {\bf Acknowledgements.} This work is supported by NSFC
(11931009, 12271265, 12261131498, 12326319, 12401032), National Postdoctoral
Program for Innovative Talents of China (BX20220158), China
Postdoctoral Science Foundation (2022M711708), Fundamental
Research Funds for the Central Universities and Nankai Zhide
Foundation. The authors thank the referees for valuable suggestions to improve the paper.

\end{document}